\documentclass[a4paper,11pt]{article}

\usepackage{amsthm, amssymb, amsmath}

\newtheorem{thm}{Theorem}[section]
\newtheorem{lemma}[thm]{Lemma}

\newtheorem{defn}[thm]{Definition}
\newtheorem{prop}[thm]{Proposition}
\newtheorem{example}[thm]{Example}
\newtheorem{notation}[thm]{Notation}
\newtheorem{observation}[thm]{Observation}

\newcommand{\N}{\mathbb{N}}
\newcommand{\Z}{\mathbb{Z}}
\newcommand{\Q}{\mathbb{Q}}
\newcommand{\R}{\mathbb{R}}

\newcommand{\M}{\mathfrak{M}}
\newcommand{\Qp}{\mathbb{Q}_p}
\newcommand{\Zp}{\mathbb{Z}_p}
\newcommand{\Fq}{\mathbb{F}_q}
\newcommand{\EE}{\mathcal{E}}

\newcommand{\ilim}{\mathop{\varprojlim}\limits}
\newcommand{\bfalpha}{{\boldsymbol{\alpha}}}
\newcommand{\bfbeta}{{\boldsymbol{\beta}}}

\newcommand{\bfs}{{\boldsymbol{s}}}
\newcommand{\bfst}{{\bf{st}}}
\newcommand{\bfone}{{\boldsymbol{1}}}
\newcommand{\m}{{\rm m}}

\title{Sharp bounds for the number of roots of univariate fewnomials}
\author{Mart\'\i n Avenda\~no\thanks{Partially supported by NSF MCS grant DMS-0915245.}\\
       \small Texas A\&M University, Department of Mathematics\\[-0.8ex]
       \small Milner Bldg.~023, College Station, TX 77843-3368, USA\\
       \small \texttt{avendano@math.tamu.edu}\\
       \and
       Teresa Krick\thanks{Research supported by grants UBACYT X-113, 2008-2010, and CONICET PIP-11220090100801, 2010-2012.}\\
       \small Departamento de Matem\'atica, FCEyN, Universidad de Buenos Aires and CONICET\\[-0.8ex]
       \small Ciudad Universitaria, --1428-- Buenos Aires, Argentina\\
       \small \texttt{krick@dm.uba.ar}\\       
}

\date{\today}

\begin{document}

\maketitle

\begin{abstract}
Let $K$ be a field and $t\geq 0$. Denote by $B_\m(t,K)$ the supremum
of the number of roots in $K^\ast$, counted with multiplicities,
that can have   a non-zero polynomial in $K[x]$ with at most $t+1$
monomial terms. We prove, using an unified approach based on
Vandermonde determinants,
 that $B_\m(t,L)\leq
t^2 B_\m(t,K)$ for any local field $L$ with a non-archimedean
valuation $v:L\to\R\cup\{\infty\}$ such that $v|_{\Z_{\neq 0}}\equiv
0$ and residue field $K$, and that $B_\m(t,K)\leq (t^2-t+1)(p^f-1)$
for any finite extension $K/\Qp$ with residual class degree~$f$ and
ramification index~$e$, assuming that $p>t+e$. For any finite
extension $K/\Qp$, for  $p$ odd, we also show the lower bound
$B_\m(t,K)\geq (2t-1)(p^f-1)$, which gives the sharp estimation
$B_\m(2,K)=3(p^f-1)$ for trinomials when $p>2+e$.
\end{abstract}



\section{Introduction}\label{sec-intro}

\begin{defn}\label{def1}
Let $K$ be a field and let $t\geq 0$. We denote by $B_1(t,K)$ and
$B_\m(t,K)$ the supremum of the  number of roots in $K^\ast$,
counted without/with multiplicities respectively, that can have a
non-zero polynomial in $K[x]$ with at most $t+1$ monomial terms.
\end{defn}

Since a  monomial can not have non-zero roots, we have
$B_1(0,K)=B_\m(0,K)=0$ for any field $K$. For this reason, we
restrict our attention to the case $t \ge 1$. Note also that
$B_1(t,K)\leq B_\m(t,K)$ for any field $K$ and any $t\geq 0$.
Moreover, if $K$ is a field of characteristic zero, it can be shown
(by taking  derivatives) that any root in $K^\ast$ of a polynomial
with $t+1$ non-zero terms has multiplicity at most $t$, hence
$B_\m(t,K)\leq t\,B_1(t,K)$, although we do not even know whether
$B_\m(t,K)$ might be greater than $B_1(t,K)$ in this case. When $K$
is a field of characteristic $p\neq 0$, the binomials
$x^{p^n}-1=(x-1)^{p^n}\in K[x]$, which have the root $x=1$ with
multiplicity $p^n$, show that $B_\m(t,K)=\infty$ for any $t\geq 1$.
Similarly, for any algebraically closed field $K$ and $t\geq 1$, we
have $B_1(t,K)=B_\m(t,K)=\infty$, since the binomials $x^d-1$ have
$d$ different roots in $K^\ast$ for any positive integer $d$ not
divisible by the characteristic of~$K$.

% Say that B_1(t,K)>=t for any t>=0 and infinite field K.

\bigskip

For the field of  real numbers~$\R$, it is well-known by
Descartes' rule of signs that $B_1(t,\R)\leq B_\m(t,\R)\leq 2t$.
Furthermore, the equality holds since this upper bound is attained
by the polynomials $(x^2-1^2)(x^2-2^2)\cdots(x^2-t^2)\in\R[x]$, that
have exactly $t+1$ non-zero terms and $2t$  simple real roots. This
result extends straightforwardly to any ordered field by the corresponding
generalization of Descartes'rule of signs (and since the same
example stays valid), see for
instance in \cite[Prop.~1.2.14]{BCR}:

\begin{thm}\label{thm1}
Let $K$ be an ordered field. Then $$B_1(t,K)=B_\m(t,K)=2t.$$
\end{thm}

Here we give a different proof of this theorem, based on generalized
Vandermonde determinants, in order to introduce the technique used
in the proof of our main results.

\bigskip

Recall that if $K$ is an ordered field, then also the
field of formal power series $K((u))$ and the field of Puiseux
series $K\{\{u\}\}=\bigcup_{n\geq 1}K((u^{1/n}))$ are ordered (by
saying that a power series is positive if and only if its first
non-zero coefficient, i.e. the one with minimum power of $u$, is
positive). Also the field of rational functions $K(u)$ can be
ordered by embedding it into $K((u))$. Theorem~\ref{thm1} implies
that $B_1(t,K\{\{u\}\})=B_\m(t,K\{\{u\}\})=2t$ for any ordered
field~$K$.

\bigskip

For other fields, the situation can be dramatically different. For
instance, B.~Poonen showed in~\cite[Thm.~1]{Poonen}, that in the
case $K=\Fq$, we have $B_1(t,\Fq\{\{u\}\})=q^t$. In the case of a
field $K$ of characteristic zero, next result gives a bound for
$B_1(t,K\{\{u\}\})$ and $B_\m(t,K\{\{u\}\})$ in terms of $B_1(t,K)$
and $B_\m(t,K)$.

\begin{thm}\label{thm2}
Let $L$ be a local field with a valuation $v:L\to\R\cup\{\infty\}$
such that $v(n\cdot 1_L)=0$ for all $n\in\Z\setminus\{0\}$, and  let
$K$ be its residue field. Then $$B_1(t,L)\leq t^2\,B_1(t,K)\ \mbox{
and } \ B_\m(t,L)\leq t^2\,B_\m(t,K).$$
\end{thm}

Note that the assumption $v(n\cdot 1_L)=0$ for all
$n\in\Z\setminus\{0\}$ implies that $L$ is a field of characteristic
zero, because otherwise we would obtain a contradiction in $v({\rm
char}(L)\cdot 1_L)=v(0)=\infty$. Also, by construction of the
residue field, we have that $v({\rm char}(K)\cdot 1_L)>0$, and hence
$K$ is also implied to be of characteristic zero. Theorem~\ref{thm2}
can be applied to the fields $L=K((u))$ or $L=K\{\{u\}\}$, as long
as a bound for $B_1(t,K)$ or $B_\m(t,K)$ is provided. The valuation
on $L$ used in this case is the trivial one, i.e. $v|_{K^\ast}=0$
and $v(u)=1$. Unfortunately, the bound obtained is not sharp in
general. For instance, the case $K=\R$ give us
$B_1(t,\R\{\{u\}\})\leq B_\m(t,\R\{\{u\}\})\leq 2t^3$, while the
sharpest bound is $2t$.

\bigskip

Theorem~\ref{thm2} can not be applied to the field $\Qp$ of $p$-adic
numbers (nor to any finite extension $K/\Qp$), since its residue
field has non-zero characteristic. In the case of a finite extension
$K$  of $\Q_p$ with ramification index $e$ and residue class degree
$f$, H.W.~Lenstra proved in~\cite[Prop.~7.2]{Lenstra} that
$$B_\m(t,K)\leq
c\,t^2\,(p^{f}-1)\,(1+e\log(e\,t/\log(p))/\log(p)),$$
$c=e/(e-1)\approx 1.58197671$. Our following result improves
Lenstra's  for   prime numbers $p$ large enough with respect to the
number of non-zero terms.

\begin{thm}\label{thm3}
Let $K/\Qp$ be a finite extension, with ramification index $e$ and
residue class degree $f$. Assume that $p>e + t$. Then
$$B_\m(t,K)\leq (t^2-t+1)\,(p^{f}-1).$$
\end{thm}

The previous bound  is  sharp for binomials (i.e.~$t=1$), since the
polynomial $x^{p^{f}}-x\in K[x]$ has $p^{f}-1$ roots in $K^\ast$. It
is also sharp for trinomials (i.e.~$t=2$) when $p>2+e$, thanks to
the following explicit example, see Section~\ref{sec-lower}.

\begin{example}\label{thm4}
Let $p$ be an odd prime number and let $K/\Q_p$ be a finite
extension with residue field of cardinality $q$. Then, the trinomial
$$f(x)=x^{(q-1)(1+q^{q-1})}-(1+q^{q-1})x^{q-1}+q^{q-1}\in K[x]$$
has at least $3(q-1)$ roots in $K^\ast$ counted with multiplicities.
\end{example}

In~\cite{AvIbr}, the authors define the class of regular polynomials
in $K[x]$, where $K$ is a local field with respect to a discrete
valuation  with residue field of cardinality $q<+\infty$, and prove
that the polynomials in this class can not have more than $t(q-1)$
roots in $K^\ast$, counted with multiplicities
\cite[Cor.~4.6]{AvIbr}. Moreover, this bound is sharp for regular
polynomials, since explicit examples (with all simple roots) are
presented. This implies the lower bound $B_\m(t,K)\geq B_1(t,K)\geq
t(q-1)$. Note that  in particular, this  lower bound holds for any
finite extension $K/\Qp$. The following result improves it in this
case by a factor of almost $2$.

\begin{thm}\label{thm99}
Let $p$ be an odd prime number and let $K/\Qp$ be a finite extension
with residue field of cardinality $q$. Then $B_1(t,K)\geq
(2t-1)(q-1)$.
\end{thm}

The previous  results also complement another result by Lenstra,
where the sharp estimate $B_\m(2,\Q_2)=6$ is
shown~\cite[Prop.~9.2]{Lenstra}. He also asks for the exact value of
$B_\m(2,\Qp)$ for other primes $p$. As a consequence of
Theorems~\ref{thm3} and \ref{thm99} and Example~\ref{thm4}  we
derive that
$$B_1(2,\Q_p)=B_\m(2,\Qp)=3(p-1)$$ for any prime $p\geq 5$,
thus leaving the case $p=3$ as the only remaining open question. For
$t=3$, we are also closing the gap to $$5(p-1)\leq B_1(3,\Qp)\leq
B_\m(3,\Qp)\leq 7(p-1)$$ for any prime $p\geq 5$. Moreover, for any
$(t+1)$-nomial over $\Qp$ with $p>t+1$ we deduce $$(2t-1)(p-1)\leq
B_1(t,\Qp)\leq B_\m(t,\Qp)\leq (t^2-t+1)(p-1).$$ A deeper analysis
for the case $t=3$  may give a hint of whether the sharp bound for
$(t+1)$-nomials is linear or quadratic in $t$. Our feeling is that
it should be quadratic although we do not have yet any evidence to
support it.

\section{Generalized confluent Vandermonde determinants}\label{sec-vander}

\begin{defn}\label{def-vander}
Let $\bfalpha=(\alpha_1,\ldots,\alpha_t)\in\N^t$ and for $s\in N$,
$(x_0,\dots, x_{s-1})$ be a group of   $s$ variables. The {\em
generalized Vandermonde matrix} associated to $\alpha$ is defined as
$$M_\bfalpha(x_0,\ldots,x_{s-1})=
       \left(\begin{array}{cccc}
           1 & x_0^{\alpha_1} & \cdots & x_0^{\alpha_t}\\
           \vdots & \vdots & & \vdots \\
           1 & x_{s-1}^{\alpha_1} & \cdots & x_{s-1}^{\alpha_t}\\
       \end{array}\right)\ \in \ \Z[x_0,\dots,x_{s-1}]^{s\times (t+1)} .$$
When $s=t+1$, the polynomial
$$V_\bfalpha(x_0,\ldots,x_t)=\det(M_\bfalpha(x_0,\ldots,x_t))\in \Z[x_0,\ldots,x_t]$$
is called a {\em generalized Vandermonde determinant}.
\end{defn}

We note that when $\bfst=(1,2,\dots,t)$, then
$V_\bfst(x_0,\dots,x_t)$ corresponds to the standard Vandermonde
determinant.

\bigskip

 The basic properties of generalized Vandermonde determinants are summarized in the
following well-known proposition, see for instance ~\cite[Thm.~5]{Evans} or~\cite{Mitchell}.

\begin{prop}\label{prop-vander}
Let $\bfalpha=(\alpha_1,\ldots,\alpha_t)$ with
$0<\alpha_1<\cdots<\alpha_t$. Then
\begin{description}
  \item[(a)] $\displaystyle{V_{\bfst}(x_0,\ldots,x_t)=\prod_{0\leq i<j\leq t}(x_j-x_i)}$.
  \item[(b)] $V_\bfalpha=V_{\bfst}P_\bfalpha$ for some non-zero $P_\bfalpha\in\Z[x_0,\ldots,x_t]$.
  \item[(c)] $V_\bfalpha$ and $P_\bfalpha$ are homogeneous polynomials of degree
  $|\bfalpha|$ and $|\bfalpha|-t(t+1)/2$ respectively.
  \item[(d)] The coefficients of $P_\bfalpha$ are all non-negative.
\end{description}
\end{prop}

We show now, before dealing with multiplicities, how Proposition~\ref{prop-vander}
immediately implies $B_1(t,K)\le 2t$ for ordered fields.

\begin{proof}[Proof of Theorem~\ref{thm1} (first part)]
Suppose that $B_1(t,K)>2t$. Then there exists a non-zero polynomial
$f=a_0+a_1x^{\alpha_1}+\cdots+a_tx^{\alpha_t}\in K[x]$  with
strictly more than $2t$ different roots. Therefore at least $t+1$ of
these roots, say $r_0, \dots, r_t$, are all strictly positive or
strictly negative. The equalities $f(r_i)=0$ for $0\le i\le t$
translate into the matrix identity
$$ M_\bfalpha (r_0,\dots,r_t)\cdot \left( \begin{array} {c} a_0\\ \vdots \\ a_t
\end{array}\right) = \left(\begin{array}{c} 0\\ \vdots \\
0\end{array}\right)$$ and since $f\ne 0$, we conclude that
$V_\bfalpha(r_0,\dots,r_t)=0$. However, by
Proposition~\ref{prop-vander}({\bf d}),
$$V_\bfalpha(r_0,\dots,r_t)=V_\bfst (r_0,\dots, r_t)\,P_\bfalpha (r_0,\dots,r_t)\ne 0$$
since $V_\bfst (r_0,\dots, r_t)\ne 0$ and $P_\bfalpha
(r_0,\dots,r_t)$ is strictly positive or negative according to  the
sign of the $r_i$'s. Contradiction!
\end{proof}

In order to deal with multiple roots, we need a more general version
of Definition~\ref{def-vander} and Proposition~\ref{prop-vander}.

\begin{defn}\label{def-vander-mult}
Let $\bfalpha=(\alpha_1,\ldots,\alpha_t)\in\N^t$, $(x_0,\dots,
x_{m})$ be a group of   $m+1$ variables for $m\ge 0$, and
$\bfs=(s_0,\dots, s_{m})\in\N^{m+1}$. The {\em generalized confluent
Vandermonde matrix} associated to $\bfalpha$ and $\bfs$ is defined
as
$$
   M^\bfs_\bfalpha(x_0,\dots,x_{m})=\stackrel{ \longleftarrow\ \ \scriptstyle{t+1} \ \ \longrightarrow}{\left(
    \begin{array}{ccc}
    & M^{s_0}_\bfalpha(x_0) &
  \\[2pt]
   \hline
  & \vdots & \\[2pt]
  \hline
  & M^{s_{m-1}}_\bfalpha(x_{m-1}) &
  \end{array}\right)}
  \begin{array}{l}
\scriptstyle s_0 \\
{\scriptstyle \ \vdots} \\
{\scriptstyle s_{m}}
  \end{array} \quad \in \ \Z[x_0,\dots, x_{m}]^{|\bfs|\times(t+1)}
$$
where for $0\le i\le m$,
$$M^{s_i}_\bfalpha(x_i)=\left(\begin{array}{cccc}
    1 & x_i^{\alpha_1} & \cdots & x_i^{\alpha_t} \\
    0 & \alpha_1x_i^{\alpha_1-1} & \cdots & \alpha_tx_i^{\alpha_t-1} \\
    \vdots & \vdots & & \vdots \\
    0 & \binom{\alpha_1}{s_i-1}x_i^{\alpha_1-s_i+1} & \cdots & \binom{\alpha_t}{s_i-1}x_i^{\alpha_t-s_i+1}
\end{array}\right) \ \in \ \Z[x_i]^{s_i\times (t+1)}.$$
When $|\bfs|=t+1$,  the polynomial
$$V_\bfalpha^\bfs(x_0,\ldots,x_m)=\det(M_\bfalpha^\bfs(x_0,\ldots,x_m))\in \Z[x_0,\ldots,x_m]$$
is called a {\em generalized confluent Vandermonde determinant}.
\end{defn}

Note that the matrix $M_\bfalpha(x_0,\ldots,x_{s-1})$ of
Definition~\ref{def-vander}  corresponds to the matrix
$M_\bfalpha^\bfone(x_0,\ldots,x_{s-1})$ with
$\bfone=(1,\ldots,1)\in\N^{s}$.

\bigskip

Next result generalizes Proposition~\ref{prop-vander} to these more
general matrices.
\begin{prop}\label{prop-vander-mult}
Let $\bfalpha=(\alpha_1,\ldots,\alpha_t)$ with
$0<\alpha_1<\cdots<\alpha_t$ and let ${\bfst}=(1,2,\ldots,t)$. Let
$\bfs=(s_0,\ldots,s_m)\in\N^{m+1}$ with $|\bfs|=t+1$. Then
\begin{description}
  \item[(a)] $\displaystyle{V_{\bfst}^\bfs(x_0,\ldots,x_m)=\prod_{0\leq i<j\leq m}(x_j-x_i)^{s_is_j}}$.
  \item[(b)] $V_\bfalpha^\bfs=V_{\bfst}^\bfs P_\bfalpha^\bfs$ for some $P_\bfalpha^\bfs\in\Z[x_0,\ldots,x_m]$.
  \item[(c)] $V_\bfalpha^\bfs$ and $P_\bfalpha^\bfs$ are homogeneous polynomials of degree $|\bfalpha|-\sum_{i=0}^m
  s_i(s_i-1)/2$
   and $|\bfalpha|-t(t+1)/2$ respectively.
  \item[(d)] The coefficients of $P_\bfalpha^\bfs$ are all non-negative.
\end{description}
\end{prop}
\begin{proof}
 Set
\begin{align*}\hat{\bfs}&=(s_0,\ldots,s_{k-1},s_k+1,s_{k+1},\ldots,s_m)\in\N^{m+1},\\
\bar{\bfs}&=(s_0,\ldots,s_{k-1},s_k,1,s_{k+1},\ldots,s_m)\in\N^{m+2}.\end{align*}
The proofs will be inductive, assuming the properties hold for
$\bar{\bfs}$ and proving them for $\hat{\bfs}$,   noting that the
case $\bfs=(1,1,\ldots,1)\in \N^{t+1}$ corresponds to
Proposition~\ref{prop-vander}. They are based on the following
identity of polynomials.
\begin{equation} \label{eq1}
V_\bfalpha^{\hat{\bfs}}(x_0,\ldots,x_m)=
  \left.\frac{V_\bfalpha^{\bar{\bfs}}(x_0,\ldots,x_k,x_k+\delta,x_{k+1},\ldots,x_m)}{\delta^{s_k}}\right|_{\delta=0}
\end{equation}
To prove Identity~(\ref{eq1}), we  perform row operations on
$M_\alpha^{\bar{\bfs}}(\ldots,x_k,x_k+\delta,\ldots)$. More
precisely, we will only operate on the subblock $A$ of this matrix
corresponding to $M_\bfalpha^{s_k}(x_k)$ and
$M_\bfalpha^{1}(x_k+\delta)$.
$$A\!=\!{\left(
    \begin{array}{ccc}
    & M^{s_k}_\bfalpha(x_k) &
  \\[2pt]
   \hline
  & M^{1}_\bfalpha(x_k+\delta) &
  \end{array}\right)}
  \!=\!  \left(\begin{array}{cccc}
1 & x_k^{\alpha_1} & \cdots & x_k^{\alpha_t} \\
0 & \alpha_1x_k^{\alpha_1-1} & \cdots & \alpha_tx_k^{\alpha_t-1} \\
\vdots & \vdots & & \vdots \\
0 & \binom{\alpha_1}{s_k-1}x_k^{\alpha_1-s_k+1} & \cdots & \binom{\alpha_t}{s_k-1}x_k^{\alpha_t-s_k+1} \\
\hline  1 & (x_k+\delta)^{\alpha_1} & \cdots &
(x_k+\delta)^{\alpha_t}
\end{array}\right)$$
Expanding the last row and subtracting the first $s_k$ rows
multiplied by $\delta^{i-1}$ from the last one, we get
$$B=\left(\begin{array}{cccc}
1 & x_k^{\alpha_1} & \cdots & x_k^{\alpha_t} \\
0 & \alpha_1x_k^{\alpha_1-1} & \cdots & \alpha_tx_k^{\alpha_t-1} \\
\vdots & \vdots & & \vdots \\
0 & \binom{\alpha_1}{s_k-1}x_k^{\alpha_1-s_k+1} & \cdots & \binom{\alpha_t}{s_k-1}x_k^{\alpha_t-s_k+1} \\
0 & \sum_{i\geq s_k}\binom{\alpha_1}{i}\delta^ix_k^{\alpha_1-i} &
\cdots & \sum_{i\geq s_k}\binom{\alpha_t}{i}\delta^ix_k^{\alpha_t-i}
\end{array}\right)$$
Now we compute the determinant
$V_\bfalpha^{\bar{\bfs}}(\ldots,x_k,x_k+\delta,\ldots)$ using the
block $B$ instead of $A$. The last row of $B$ shows that it is
divisible by $\delta^{\bfs_k}$. Moreover, dividing by $\delta^{s_k}$
and then specializing it into $\delta=0$ corresponds to  keeping
only the term in $\delta^{s_k}$ in the last row of $B$, thus
reducing to the determinant of the matrix
$M^{\hat{\bfs}}_\bfalpha(x_0,\ldots,x_m)$. This concludes the proof
of Identity~(\ref{eq1}).

\smallskip
\noindent {\bf (a)} See also \cite{Aitken}  or \cite[Thm.~2.4]{ChenLi}.

\noindent Assume it holds for $\bar{\bfs}$. Then
{\small
\begin{align*}
V_{\rm st}^{\bar{\bfs}}(\ldots,x_k,&x_k+\delta,\ldots)=\!\!\!\!\!\prod_{0\leq
i<j\le m}\!\!(x_j-x_i)^{s_is_j} \!\!\!\!\prod_{0\leq i\leq
k}\!\!(x_k+\delta-x_i)^{s_i} \!\!\!\!\prod_{k<j\leq
 m}\!\!(x_j-(x_k+\delta))^{s_j}\\
&=\delta^{s_k}\!\!\!\!\prod_{0\leq i<j\le m}\!\!(x_j-x_i)^{s_is_j}
\!\!\!\!\prod_{0\leq i< k}\!\!(x_k+\delta-x_i)^{s_i}
\!\!\!\!\prod_{k<j\leq
 m}\!\!(x_j-(x_k+\delta))^{s_j}
\end{align*}}
Therefore, by Identity~(\ref{eq1}),
\begin{align*}
V_{\bfst}^{\hat{\bfs}}(x_0, & \ldots ,x_m)=\!\!\!\!\prod_{0\leq
i<j\le m}\!\!(x_j-x_i)^{s_is_j}
 \!\!\!\!\prod_{0\leq i< k}\!\!(x_k-x_i)^{s_i}
 \!\!\!\!\prod_{k<j\leq m}\!\!(x_j-x_k)^{s_j}\\
 &=\!\!\!\!\prod_{0\leq
i<j\le m; \,i,j\ne k}\!\!(x_j-x_i)^{s_is_j}
 \!\!\!\!\prod_{0\leq i< k}\!\!(x_k-x_i)^{s_i(s_k+1)}
 \!\!\!\!\prod_{k<j\leq m}\!\!(x_j-x_k)^{s_j(s_k+1)},
\end{align*}
proving that it holds for $\hat{\bfs}$.

\smallskip
\noindent {\bf (b)}  Assume it holds for  $\bar{\bfs}$. Then, by
Identity~(\ref{eq1}) and the inductive hypothesis, we get
$$V_\bfalpha^{\hat{\bfs}}(x_0,\ldots,x_m)=\left.
\frac{V_{\rm
st}^{\bar{\bfs}}(\ldots,x_k,x_k+\delta,\ldots)P_\bfalpha^{\bar{\bfs}}(\ldots,x_k,x_k+\delta,\ldots)}
{\delta^{s_k}}\right|_{\delta=0}$$ which by the previous item and
Identity~(\ref{eq1}) again
 gives
$$V_\bfalpha^{\hat{\bfs}}(x_0,\ldots,x_m)=V_{\bfst}^{\hat{\bfs}}(x_0,\ldots,x_m)P_\bfalpha^{\bar{\bfs}}
(x_0,\ldots,x_k,x_k,\ldots,x_m).$$ We conclude by setting
$P_\bfalpha^{\hat{\bfs}}(x_0,\ldots,x_m)=P_\bfalpha^{\bar{\bfs}}(x_0,\ldots,x_k,x_k,\ldots,x_m)$,
which  belongs to $\Z[x_0,\ldots,x_m]$ since
$P_\bfalpha^{\bar{\bfs}}$ has integer coefficients.

\smallskip
\noindent {\bf (c)} The proof of item {\bf (b)} shows that the
polynomials $P^\bfs_\bfalpha$ are homogeneous and  of the same
degree, independent from $\bfs$, than the polynomial $P_\bfalpha$ of
Proposition~\ref{prop-vander}.  We compute the degree of
$V_{\bfst}^\bfs$ using item {\bf (a)}:
\begin{align*}\deg(V_{\bfst}^\bfs)&=\!\!\!\!\!\sum_{0\leq i<j\leq
m}\!\!\!\!s_is_j=
  \frac{1}{2}\left(|\bfs|^2-\sum_{i=0}^ms_i^2\right)=\!
  \frac{1}{2}\left(|\bfs|^2-|\bfs|-\sum_{i=0}^ms_i(s_i-1)\right)\\
  & = \frac{t(t+1)}{2} - \frac{1}{2}\sum_{i=0}^ms_i(s_i-1).\end{align*}
Therefore, by {\bf (b)},  $V_\bfalpha^\bfs$ is a homogeneous
polynomial and
$$\deg(V_{\bfalpha}^\bfs)=\deg(V_{\bfst}^\bfs)+\deg(P_\bfalpha^\bfs)=|\bfalpha|-\sum_{i=0}^ms_i(s_i-1)/2.$$

\noindent {\bf (d)} The proof of Item~{\bf (b)} also shows that if
we assume that the polynomial $P_\alpha^{\bar{\bfs}}$ has
non-negative coefficients, then $P_\alpha^{\hat{\bfs}}$ has
non-negative coefficients as well.
\end{proof}

At this point, we have all the ingredients to prove that
$B_\m(t,K)\le 2t$ for ordered fields.

\begin{proof}[Proof of Theorem~\ref{thm1} (second part)]
Suppose that $B_\m(t,K)>2t$. Then there exists a non-zero polynomial
$f=a_0+a_1x^{\alpha_1}+\cdots+a_tx^{\alpha_t}\in K[x]$  with
strictly more than $2t$  roots counted with multiplicities. Choose,
for some $m\ge 0$,  $m+1$ of these roots, different, say
$r_0,\dots,r_m$, all strictly positive or strictly negative
satisfying  that for some $s_i\le {\rm mult}(f;r_i)$, $s_0+\cdots
+s_m= t+1$ holds, and set $\bfs=(s_0,\dots,s_m)$. Note that since
char$(K)=0$, the equalities
$$f(r_i)=\dots = f^{(s_i-1)}(r_i)=0 \ \mbox{ for } 0\le i\le m$$
translate into the matrix identity
$$ M_\bfalpha^\bfs (r_0,\dots,r_m)\cdot \left( \begin{array} {c} a_0\\ \vdots \\ a_t
\end{array}\right) = \left(\begin{array}{c} 0\\ \vdots \\
0\end{array}\right).$$ This is because the $k$-th row of
$M^{s_i}_\bfalpha(r_i) $ times $(a_0, \dots ,a_t)^t$ equals
$\frac{f^{(k-1)}(r_i)}{(k-1)!}$.

\noindent Since $f\ne 0$, we conclude that $V_\bfalpha^\bfs(r_0,\dots,r_m)=0$.
However, by Proposition~\ref{prop-vander-mult}{\bf (d)},
$$V_\bfalpha^\bfs(r_0,\dots,r_m)=V_\bfst^\bfs (r_0,\dots, r_m)\,P_\bfalpha ^\bfs(r_0,\dots,r_m)\ne 0$$
since $V_\bfst ^\bfs(r_0,\dots, r_m)\ne 0$ and $P_\bfalpha^\bfs (r_0,\dots,r_m)$
is strictly positive or negative according to the sign of the $r_i$'s.
Contradiction!
\end{proof}

Note that the proof of Proposition~\ref{prop-vander-mult}{\bf (b)}
shows inductively that
\begin{equation}
P_\bfalpha^\bfs(x_0,\ldots,x_m)=
P_\bfalpha(\underbrace{x_0,\ldots,x_0}_{s_0},\underbrace{x_1,\ldots,x_1}_{s_1},\ldots,\underbrace{x_m,\ldots,x_m}_{s_m}).
\label{eq2}
\end{equation}
This observation is useful for the proof of next result, which will
be used in Section~\ref{sec-roots1}.
\begin{lemma}\label{formula-vander1}
Let $\bfalpha=(\alpha_1,\ldots,\alpha_t)$ with
$0<\alpha_1<\cdots<\alpha_t$ and let
$\bfs=(s_0,\ldots,s_m)\in\N^{m+1}$. Then
\begin{equation*}
P_\bfalpha^\bfs(1+x_0,\ldots,1+x_m)=\!\!\!\sum_{1\leq
\beta_1<\cdots<\beta_t}\!\!\!\det\left(\begin{array}{ccc}
\binom{\alpha_1}{\beta_1} & \cdots & \binom{\alpha_1}{\beta_t} \\
\vdots & & \vdots \\
\binom{\alpha_t}{\beta_1} & \cdots & \binom{\alpha_t}{\beta_t}
\end{array}\right)P_\bfbeta^\bfs(x_0,\ldots,x_m)
\end{equation*}
where $\bfbeta=(\beta_1,\ldots,\beta_t)$. The same formula holds
when replacing $P$ by $V$.
\end{lemma}
\begin{proof} First we prove the identity  for $V_\bfalpha$.
\begin{align*}V_{\bfalpha}(1+x_0,\ldots,1&+x_t)=\det\left(
  \begin{array}{cccc}
    1 & (1+x_0)^{\alpha_1} & \cdots & (1+x_0)^{\alpha_t} \\
    1 & (1+x_1)^{\alpha_1} & \cdots & (1+x_1)^{\alpha_t} \\
    \vdots & \vdots & & \vdots \\
    1 & (1+x_t)^{\alpha_1} & \cdots & (1+x_t)^{\alpha_t}
  \end{array}\right)\\ =&\det\left(
  \begin{array}{cccc}
   1 & \sum_{\beta_1\geq 0}\binom{\alpha_1}{\beta_1}x_0^{\beta_1} & \cdots & \sum_{\beta_t\geq 0}\binom{\alpha_t}{\beta_t}x_0^{\beta_t} \\
   1 & \sum_{\beta_1\geq 0}\binom{\alpha_1}{\beta_1}x_1^{\beta_1} & \cdots & \sum_{\beta_t\geq 0}\binom{\alpha_t}{\beta_t}x_1^{\beta_t} \\
   \vdots & \vdots & & \vdots \\
   1 & \sum_{\beta_1\geq 0}\binom{\alpha_1}{\beta_1}x_t^{\beta_1} & \cdots & \sum_{\beta_t\geq 0}\binom{\alpha_t}{\beta_t}x_t^{\beta_t}
  \end{array}\right)\\
=&\sum_{\beta_1,\ldots,\beta_t\geq
1}\binom{\alpha_1}{\beta_1}\cdots\binom{\alpha_t}{\beta_t}\det\left(
  \begin{array}{cccc}
   1 & x_0^{\beta_1} & \cdots & x_0^{\beta_t} \\
   1 & x_1^{\beta_1} & \cdots & x_1^{\beta_t} \\
   \vdots & \vdots & & \vdots \\
   1 & x_t^{\beta_1} & \cdots & x_t^{\beta_t}
  \end{array}\right)
\end{align*}

\noindent Reducing our sum to $\beta_1,\ldots,\beta_t$ pairwise different, and
using the definition of determinant in the last line, we get
{\small
\begin{align*}V_{\bfalpha}(1+x_0,&\ldots,1+x_t) \ = \ \\
=&
\sum_{\genfrac{}{}{0pt}{}{1\leq\beta_1<\cdots<\beta_t}{\sigma\in{\rm
Perm}\{1,\ldots,t\}}}
  \binom{\alpha_1}{\beta_{\sigma(1)}}\cdots\binom{\alpha_t}{\beta_{\sigma(t)}}\det\left(
  \begin{array}{cccc}
   1 & x_0^{\beta_{\sigma(1)}} & \cdots & x_0^{\beta_{\sigma(t)}} \\
   1 & x_1^{\beta_{\sigma(1)}} & \cdots & x_1^{\beta_{\sigma(t)}} \\
   \vdots & \vdots & & \vdots \\
   1 & x_t^{\beta_{\sigma(1)}} & \cdots & x_t^{\beta_{\sigma(t)}}
  \end{array}\right)\\
=&\sum_{\genfrac{}{}{0pt}{}{1\leq\beta_1<\cdots<\beta_t}{\sigma\in{\rm
Perm}\{1,\ldots,t\}}}
  (-1)^{|\sigma|}\binom{\alpha_1}{\beta_{\sigma(1)}}\cdots\binom{\alpha_t}{\beta_{\sigma(t)}}\det\left(
  \begin{array}{cccc}
   1 & x_0^{\beta_1} & \cdots & x_0^{\beta_t} \\
   1 & x_1^{\beta_1} & \cdots & x_1^{\beta_t} \\
   \vdots & \vdots & & \vdots \\
   1 & x_t^{\beta_1} & \cdots & x_t^{\beta_t}
  \end{array}\right)\\
 =&\sum_{1\leq \beta_1<\cdots<\beta_t}\!\!\!\det\left(\begin{array}{ccc}
\binom{\alpha_1}{\beta_1} & \cdots & \binom{\alpha_1}{\beta_t} \\
\vdots & & \vdots \\
\binom{\alpha_t}{\beta_1} & \cdots & \binom{\alpha_t}{\beta_t}
\end{array}\right)V_\bfbeta(x_0,\ldots,x_t).
\end{align*}}

\noindent Now note that by Proposition~\ref{prop-vander-mult}{\bf (a)},
$$V_{\bfst}^\bfs(1+x_0,\ldots,1+x_m)=V_{\bfst}^\bfs(x_0,\ldots,x_m).$$
Therefore, the identity holds for $P_\bfalpha$ by
Proposition~\ref{prop-vander}{\bf (b)}. Next, Identity~(\ref{eq2})
implies that the identity holds for $P_\bfalpha^\bfs$, and finally
the identity for $V_\bfalpha^\bfs$ follows from
Proposition~\ref{prop-vander-mult}{\bf (b)}.
\end{proof}

Lemma~\ref{formula-vander1} motivates the need of working with
determinants of matrices whose terms are binomial coefficients. The
following notation and results show that they share many properties
with the generalized Vandermonde determinants.

\begin{notation}\label{def-binom}
Let $\bfbeta=(\beta_1,\ldots,\beta_t)\in\Z_{\geq 0}^t$. We set
$$W_\bfbeta(x_1,\dots,x_t)=\det\left(\begin{array}{ccc}
\binom{x_1}{\beta_1} & \cdots & \binom{x_1}{\beta_t} \\
\vdots & & \vdots \\
\binom{x_t}{\beta_1} & \cdots & \binom{x_t}{\beta_t}
\end{array}\right)\ \in\ \Q[x_1,\ldots,x_t]$$
where $\binom{x}{\beta}=x(x-1)\cdots(x-\beta+1)/\beta!$.
\end{notation}

\begin{lemma}\label{prop-binom}
Let $1\leq\beta_1<\cdots<\beta_t$ and let ${\bfst}=(1,2,\ldots,t)$.
Then
\begin{description}
\item[(a)] $\beta_1!\cdots\beta_t!\,W_\beta(x_1,\ldots,x_t)\ \in\ \Z[x_1,\ldots,x_t]$.
\item[(b)] $\displaystyle{1!2!\cdots t!\,W_{\bfst}(x_1,\ldots,x_t)=x_1\cdots x_t\prod_{1\leq i<j\leq t}(x_j-x_i)}$.
\item[(c)] $\beta_1!\cdots\beta_t!\,W_\bfbeta=1!2!\cdots t!\,W_{\bfst}\,Q_\bfbeta$ for some non-zero
$Q_\bfbeta \in\Z[x_1,\ldots,x_t]$.
\item[(d)] $\deg(W_\bfbeta)=|\bfbeta|$ and $\deg(Q_\bfbeta)=|\bfbeta|-t(t+1)/2$.
\end{description}
\end{lemma}
\begin{proof}
{\bf (a)} Multiply the $j$-th column of the matrix by $\beta_j!$.

\smallskip
\noindent {\bf (b)} We need to prove that
$$\det\left(\begin{array}{ccc}
\binom{x_1}{1} & \cdots & \binom{x_1}{t} \\
\vdots & & \vdots \\
\binom{x_t}{1} & \cdots & \binom{x_t}{t}
\end{array}\right)=\frac{x_1\cdots x_t\,\prod_{1\leq i<j\leq t}(x_j-x_i)}{1!2!\cdots t!}.$$
Taking $x_i$ as a common factor in the $i$-th row and $1/j!$ as a
common factor in the $j$-th column, this determinant equals
$$\frac{x_1\cdots x_t}{1!2!\cdots t!}\det\!\!\left(\!\!
\begin{array}{ccccc}
1 & x_1-1 & (x_1-1)(x_1-2) & \cdots & (x_1-1)\cdots(x_1-t+1) \\
1 & x_2-1 & (x_2-1)(x_2-2) &\cdots & (x_2-1)\cdots(x_2-t+1) \\
\vdots & \vdots & \vdots & & \vdots \\
1 & x_t-1 & (x_t-1)(x_t-2) & \cdots & (x_t-1)\cdots(x_t-t+1)
\end{array}
\!\!\right)\!\!.$$
It is clear that adding the first to the
second column, we get $(x_1,\ldots,x_t)^t$ in the second column. Next,
adding a combination of the first and the new second column to the
third, we get $(x_1^2,\ldots,x_t^2)^t$ in the third column, etc.
Therefore our determinant equals
$$\frac{x_1\cdots x_t}{1!2!\cdots t!}\det\left(\!\!
\begin{array}{ccccc}
1 & x_1 & x_1^2 & \cdots & x_1^{t-1} \\
1 & x_2 &x_2^2 & \cdots & x_2^{t-1} \\
\vdots & \vdots & \vdots & & \vdots \\
1 & x_t & x_t^2 & \cdots & x_t^{t-1}
\end{array}
\right)$$ which shows the statement.

\smallskip
\noindent {\bf (c)} The polynomial $W_\bfbeta(x_1,\ldots,x_t)$ is
divisible by $x_1\cdots x_t$, since setting $x_i=0$ in the matrix
that defines it yields a column of  zeros. Similarly, it is
divisible by the binomials $x_j-x_i$, since setting $x_i=x_j$ would
produce two identical columns. Since the polynomial
$\beta_1!\cdots\beta_t!W_\bfbeta \in \Z[x_1,\ldots,x_t]$ of item~{\bf (a)}
is divisible by all these coprime and monic factors,  the quotient
$Q_\bfbeta$ has integer coefficients.

\smallskip
\noindent {\bf (d)}  Since the degree of the
$j$-th column of the matrix defining $W_\bfbeta$ equals $\beta_j$, then
$\deg(W_\bfbeta)\leq |\bfbeta|$. Moreover, a simple inspection shows
that the monomial $x_1^{\beta_1}\cdots x_t^{\beta_t}$ can not be
canceled. This means that $\deg(W_\bfbeta)=|\bfbeta|$, and therefore,
by item~{\bf (c)}, we conclude $\deg(Q_\bfbeta)=|\bfbeta|-t(t+1)/2$.
\end{proof}

\begin{observation}\label{obsW} When $\bfalpha=(\alpha_1,\dots,\alpha_t) \in
\Z_{\ge 0}^t$, then $W_\bfbeta(\bfalpha) \in \Z$, since in this
case, all the entries of the matrix defining it are integer numbers.
\end{observation}

\section{Local fields}\label{sec-roots1}

Throughout this section we assume that $L$ is a local field of
characteristic zero  with respect to the non-archimedean valuation
$v:L\to\R\cup\{\infty\}$. The ring of integers~$A=\{x\in
L\,:\,v(x)\geq 0\}$ is a local ring with maximal ideal $\M=\{x\in
L\,:\,v(x)>0\}$. The residue field of $L$ is the quotient $K=A/\M$.

\begin{defn}\label{def2}
Let $t\geq 0$. We denote by $D_1(t,L)$ and $D_m(t,L)$ the supremum
of the number of roots in $1+\M$, counted without/with
multiplicities respectively, that can have a non-zero polynomial in
$L[x]$ with at most $t+1$ non-zero terms.
\end{defn}

Next example, which shows that the bound $D_1(t,L)=\infty$ can be
reached, was suggested to us by the referee: take
$L=\overline{\Q_p}$ and $f=x^{p^n}-1$. The set of solutions of this
polynomial is  the cyclic group of order $p^n$, which is indeed
contained in  $1+\M$, since by Fermat's theorem, $r^{p-1}\in 1+\M$
for such a root, and $p-1$ is prime to $p$.

\smallskip
Note that $D_1(t,L)\leq D_\m(t,L)\leq t\,D_1(t,L)$, since in
characteristic zero the roots of a polynomial with $t+1$ terms can
not have multiplicity greater than $t$.

\begin{prop}\label{lem3}
Let $L$ be a field with a valuation $v:L\to\R\cup\{\infty\}$, with
residue field $K$. Then $$B_1(t,L)\leq t\,B_1(t,K)\,D_1(t,L) \
\mbox{ and } \ B_\m(t,L)\leq t\,B_1(t,K)\,D_\m(t,L).$$
\end{prop}
\begin{proof}
Let $f\in L[x]$ be a non-zero polynomial with at most $t+1$ terms.
The theory of Newton polygons (see~\cite[Prop.~3.1.1]{Wei}) shows
that the set $V=\{v(r)\,:\, f(r)=0, \ r\in L^\ast\}$ corresponds to
slopes of the segments of the Newton polygon $NP(f)$ of $f$ and thus
has at most $t$ elements. Take $v\in V$ and let $r_0\in L^\ast$ such
that $v(r_0)=v$. Every root $r$ of $f$ with   $v(r)=v$ corresponds
to the root $r/r_0$ of $g(x):=f(x\,r_0)$ with $v(r/r_0)=0$, with the
same multiplicity.

\smallskip
\noindent Therefore we only need to prove that $g$ has at most
$B_1(t,K)D_1(t,L)$ (resp. $B_1(t,K)D_\m(t,L)$) roots with valuation
zero counted without (resp. with) multiplicities.

\smallskip
\noindent By dividing $g$ by its coefficient with minimum valuation,
we can assume, without loss of generality, that $g\in A[x]$ and that
not all coefficients of~$g$ belong to $\M$. Let $\bar{g}\in K[x]$ be
the non-zero polynomial obtained by reducing the coefficients of~$g$
modulo~$\M$. Then, by Definition~\ref{def1}, the set $W=\{\bar r\in
K^\ast\,:\,\bar{g}(\bar r)=0\}=\{\bar r_1,\dots, \bar r_m\}$ has $m\le B_1(t,K)$
elements, each of them represented by some $r_i\in A\setminus\M$.

\smallskip
\noindent  Each root $r\in L$ of $g$ with valuation zero belongs to
some coset $r_i+\M$,  and each root of  $g$ in $r_i+\M$ corresponds
to a root of $h_i(x):=g(x\,r_i)$ in $1+\M$. Since $h_i$ has at most
$D_1(t,L)$ (resp.~$D_\m(t,L)$) roots in $1+\M$ counted without
(resp.~with) multiplicities, then $g$ has at most $mD_1(t,L)$
(resp.~$mD_\m(t,L)$) roots in $L$ with valuation zero counted
without (resp.~with) multiplicites.
\end{proof}

Now we derive Theorem~\ref{thm2} as an immediate consequence
of  Lemma~\ref{lem3} above and Proposition~\ref{thm5} below.

\begin{prop}\label{thm5}
Let $L$ be a local field with a non-archimedean valuation
$v:L\to\R\cup\{\infty\}$ such that $v(n\cdot 1_L)=0$ for all
$n\in\Z\setminus\{0\}$. Then $D_1(t,L)\leq D_\m(t,L)\leq t$.
\end{prop}
\begin{proof} The proof goes as the proof of Theorem~\ref{thm1}.
Let $A$ be the ring of integers of $L$ and let $\M$ be the maximal
ideal of $A$. As we pointed out in the introduction, the assumption
on $v$ implies that $L$ has characteristic zero. \\Suppose that
$D_\m(t,L)>t$. Then there exists a non-zero polynomial
$f=a_0+a_1x^{\alpha_1}+\cdots+a_tx^{\alpha_t}\in L[x]$  with
strictly more than $t$  roots in $1+\M$  counted with
multiplicities. Choose, for some $m\ge 0$,  $m+1$ of these roots,
different,  say $r_0,\dots,r_m$, satisfying  that for some $s_i\le
{\rm mult}(f;r_i)$, $s_0+\cdots +s_m= t+1$ holds, and set
$\bfs=(s_0,\dots,s_m)$. The equalities
$$f(r_i)=\dots = f^{(s_i-1)}(r_i)=0 \ \mbox{ for } 0\le i\le m$$
translate into the matrix identity
$$ M_\bfalpha^\bfs (r_0,\dots,r_m)\cdot \left( \begin{array} {c} a_0\\ \vdots \\ a_t
\end{array}\right) = \left(\begin{array}{c} 0\\ \vdots \\
0\end{array}\right).$$ Therefore, since $f\ne 0$, we conclude that
$V_\bfalpha^\bfs(r_0,\dots,r_m)=0$. This implies, by
Proposition~\ref{prop-vander-mult}{\bf (a-b)}, that
$P_\bfalpha^\bfs(r_0,\dots,r_m) =0$. Write $r_i=1+ x_i$ with $x_i\in
\M$ for $0\le i\le m$. Then, applying Lemma~\ref{formula-vander1}
and using Notation~\ref{def-binom},
$$P_\bfalpha^\bfs(1+x_0,\dots,1+x_m)= \!\!\!\sum_{1\leq
\beta_1<\cdots<\beta_t}W_\bfbeta (\bfalpha)
\,P_\bfbeta^\bfs(x_0,\ldots,x_m).$$ Let us show that  the term
corresponding to $\bfbeta=\bfst=(1,2,\ldots,t)$ in the right-hand
side is a non-zero integer:  $W_\bfst(\bfalpha) \in \Z$ by
Observation~\ref{obsW}, and is non-zero by
Proposition~\ref{prop-binom}{\bf (b)} since $r_i\ne r_j$; also
$P_{\bfst}^{\bfs}=1_L$ by definition. Therefore by assumption it has
valuation zero. The remaining non-zero terms have positive valuation
since in that case $W_\bfbeta (\bfalpha)$ is a non-zero integer
number, and $P_{\bfbeta}^\bfs(x_0,\ldots,x_s)$ has positive
valuation since $v(x_i)>0$ and $P_\bfbeta^\bfs$ is, according to
Proposition~\ref{prop-vander-mult}{\bf (b-c)}, a homogeneous
polynomial of positive degree with integer coefficients. Therefore
$v(P_\bfalpha^\bfs(r_0,\ldots,r_s))=0$ which implies
$P_\bfalpha^\bfs(r_0,\ldots,r_s)$ is a unit in $A$, and in particular $\ne 0$.
This contradicts the assumption $D_\m(t,L)>t$.
\end{proof}

\begin{proof}[Proof of Theorem~\ref{thm2}]
\begin{align*}B_1(t,L)& \leq \, t\,B_1(t,K)\,D_1(t,L) \qquad \mbox{by
Lemma~\ref{lem3}}\\
&\le \, t^2\,B_1(t,K)\qquad \mbox{by
Proposition~\ref{thm5}}.\end{align*} The same proof holds for
$B_\m(t,L)$.
\end{proof}

Our next aim is to prove  Theorem~\ref{thm3}. We do it following
the same lines of the proof of Theorem~\ref{thm2}, i.e. proving
first in Proposition~\ref{thm9}  below that $D_1(t,L)\leq D_\m(t,L)\leq t$ using
Lemma~\ref{formula-vander1} (which will require the extra
assumption $p>e+t$), and then using Proposition~\ref{lem6} below, that improves
Proposition~\ref{lem3} (if we used  Proposition~\ref{lem3}  we would conclude
that $B_1(t,L)\leq B_\m(t,L)\leq t^2(q-1)$ instead).

\smallskip
\medskip
In what follows, $K$ is assumed to be a  finite extension of $\Qp$
for an odd prime number $p$, with ramification index~$e$ and residue
class degree~$f$, $A$ is its ring of integers and $\M$ its maximal
ideal. The valuation $v:K\to\R\cup\{\infty\}$ of $K$ extends the
standard $p$-adic valuation $v_p$ of $\Qp$. It satisfies $v(p)=1$
and its group of values is $v(K^\times)=\frac{1}{e}\Z$. The ideal
$\M$ of $A$ is principal, generated by an element $\pi\in A$ with
valuation $v(\pi)=1/e$. The residue field $\Fq\approx A/\M$ is a
finite field of cardinality $q=p^f$. We finally define the ``first
digit" of any $x\in K^*$ to be  the first digit in its expansion,
i.e. corresponding to $\overline{\pi^{-ev(x)}x} \ \in A/\M$.

\smallskip
We will need the following lemma, which actual proof, simpler than
our previous one, was suggested by the referee.
\begin{lemma}\label{raicesunidad}
Assume that $p-1\nmid e$. Then $1$ is the only $p$-th root of unity
in $K$.
\end{lemma}
\begin{proof}
Let $\xi_p$  be a primitive $p$-root of unity. The prime $p$ is
totally ramified in $\Q(\xi_p)$, see e.g. \cite{marcus}, and
therefore the extension $\Q_p(\xi_p)/\Q_p$ has degree $p-1$. If
$K/\Q_p$ is a finite extension such that $\xi_p\in K$, we have
$\Q_p(\xi_p)\subset K$ and by the multiplicativity of the
ramification degree, $p-1\mid e$.
%Assume $\xi\in K$ is such that $\xi^p=1$ and $\xi\neq 1$. Then it is
%clear that $v(\xi)=0$, i.e. $\xi$ is a unit in $A$, that can be
%written as   $\xi=u+\pi \, x$ for some unit $u$ corresponding to a
%non-zero first digit $\overline u\in A/\M$ and some  $x\in A$.
%Therefore $\overline u^p=1 \in A/\M\approx \Fq $, and $\overline
%u^{q-1} =1$ since $\Fq^\times $ is cyclic with $q-1$ elements. Since
%$p$ and $q-1$ are coprime, this implies $\overline u=1$, i.e.
%$\xi=1+\pi \, x$ for some non-zero $x\in A$.%

%\smallskip \noindent
%Now, the condition $\xi^p=1$ implies
%$$0=\frac{(1+\pi x)^p-1}{\pi \, x}=\sum_{i=1}^{p}\binom{p}{i}\pi^{i-1} x^{i-1}= p+
%\sum_{i=1}^{p-2}\binom{p}{i+1}\pi^{i} x^{i} + \pi^{p-1}x^{p-1}.$$
%However
%\begin{itemize}
%\item $v(p)=1$,
%\item $v(\binom{p}{i+1}\pi^{i} x^{i})=1+ i/e + i\,v(x) >1$ \quad for $1\le
%i\le p-2$,
%\item $v(\pi^{p-1}x^{p-1})=(p-1)/e + (p-1)v(x) > 1 $,
%\end{itemize}
%imply that $v\big(\frac{(1+\pi x)^p-1}{\pi \, x}\big) =1$ and
%therefore $\frac{(1+\pi x)^p-1}{\pi \, x}\ne 0$.
\end{proof}

%\begin{observation} It is possible to improve the previous lemma,
%replacing the  condition $p-1>e$  by $p-1\nmid e$, although  the
%proof becomes more technical  since it requires the computation of
%the Newton polygon of $\frac{(1+\pi x)^p-1}{\pi \,x}$ as a
%polynomial in $x$.
%\end{observation}

We also need Hensel's lemma in its Newton method version, see~\cite[Prop.~3.1.2]{Wei}:

\begin{lemma}[Newton's method]\label{Hensel}
Let $K$ be a complete field with respect to a discrete
non-archimedian valuation~$v$ and let $A$ be its valuation ring. Let
$f\in A[x]$ be a non-zero polynomial and let $r_0\in A$ be such that
$v(f(r_0))>2v(f'(r_0))$. Then, there exists a unique $r\in A$ such
that $f(r)=0$ and $v( r-r_0)\ge v(f(r_0))-v(f'(r_0))> v(f'(r_0))$.
%In this case, $v(a-a_0)\geq v(f(a_0))-v(f'(a_0))>v(f'(a_0))$.
\end{lemma}

Any $r_0\in A$ satisfying the hypothesis of Lemma~\ref{Hensel} is
called an {\em approximate root} of $f$. The corresponding root
$r\in A$ of $f$ can be obtained as the limit $
r=\lim_{n\to\infty}r_n$ of the sequence given by Newton's iteration
$r_{n+1}=r_n-f(r_n)/f'(r_n)$. We also have that
$v(f'(r))=v(f'(r_0))\neq\infty$ and therefore $r$ is always a simple
root of $f$.

\begin{lemma} \label{divis} Under the same notations of Lemma~\ref{Hensel},
let $$f=a_tx^{\alpha_t}+\cdots+a_1x^{\alpha_1}+a_0\in K[x] \ \mbox{
with }  \   0<\alpha_1<\cdots<\alpha_t.$$ Assume that
$p\,\nmid\,\alpha_{i+1}-\alpha_i$ for some $i$, $0\le i \le t-1$,
and that the segment defined  by  $(\alpha_i,v(a_i))$ and
$(\alpha_{i+1},v(a_{i+1}))$ is one of the segments of the  Newton
polygon $NP(f)$ of $f$. Let
$-m_i:=(v(a_{i+1})-v(a_i))/(\alpha_{i+1}-\alpha_i)$ denote its
slope. Then, the roots of $f$ in $ K^*$ that have valuation $m_i$
are all simple and are in one-to-one correspondence with the roots
of the binomial
$$g_i=a_ix^{\alpha_i}+a_{i+1}x^{\alpha_{i+1}}.$$
Moreover, the number
of roots of $g_i$ in $K^*$ equals $\gcd(q-1,\alpha_{i+1}-\alpha_i)$
when $e\,m_i\in\Z$ and the first digit of  $a_{i+1}/a_i$ is a
$(\alpha_{i+1}-\alpha_i)$-th power in $A/\M$, or zero otherwise. In
particular, the number of roots of $f$ in $K^*$ with valuation $m_i$
is bounded by  $q-1$.
\end{lemma}
\begin{proof}
Note that  any non-zero root of  $g_i$ has necessarily valuation
$m_i$. If $e\,m_i\not\in\Z$ then there are no elements in $K^\ast$
with valuation $m_i$, i.e. no roots in $K^\ast$ of $f$ or $g_i$ with
valuation $m_i$. Let us then assume that $e\,m_i\in \Z$.

\smallskip \noindent
By making the change of variables $x\leftarrow\pi^{em_i}x$ in $f$
and $g_i$ we can reduce the proof to the case $m_i=0$, i.e.
$v(a_i)=v(a_{i+1})$ and $v(a_j)> v(a_i)$ for all $j\ne i,i+1$. By
dividing $f$ by $a_{i+1}$, we can then reduce the proof to the case
$f\in A[x]$, $a_{i+1}=1$ and $v(a_i)=0$. In particular if $g=
 a_i x^{\alpha_i}+x^{\alpha_{i+1}} $ then $f-g \in
\M[x]$.

\smallskip \noindent In this case we will show that the roots of $f$ with
valuation zero are approximate roots of $g$ and viceversa.

\smallskip \noindent
Let $r\in K^\ast$ be such that  $f(r)=0$ and $v(r)=0$. Then
$g(r)=f(r)-(f-g)(r)\in\M$, i.e. $v(g(r))>0$. Besides, since $p\nmid
\alpha_{i+1}-\alpha_i$, then
$$g'(r)=\underbrace{(\alpha_{i+1}-\alpha_i)r^{\alpha_{i+1}-1}}_{v(\cdot)=0} +
 \underbrace{\alpha_i r^{-1}g(r)}_{v(\cdot)>0}$$
has valuation zero. This means that $v(g(r))>2v(g'(r))$ and by
Lemma~\ref{Hensel}, $r$ is an approximate root of $g$.

\smallskip \noindent  Now let $r\in K^*$ be such that  $g(r)=0$, i.e. $r^{\alpha_{i+1}-\alpha_i}=-a_i$ and
therefore, since $v(a_i)=0$, $v(r)=0$.  Therefore, like above,
$v(g'(r))=0$. Also $f(r)=(f-g)(r)+g(r)$ implies $f(r)\in \M$, i.e.
${v(f(r))>0}$. Besides,
$$f'(r)=\underbrace{(f-g)'(r)}_{v(\cdot)>0}+\underbrace{g'(r)}_{v(\cdot)=0}$$
has valuation zero. Therefore $r$ is an approximate root of $f$.
This shows that there are the same number of roots, that are all
simple.

\smallskip \noindent
If the first digit of $a_i$  is not an $(\alpha_{i+1}-\alpha_i)$-th
power in $A/\M$, then clearly the binomial $g(x)$ has no roots (not
even modulo $\M$). When it is  a power, then the number of roots of
$g$ modulo $\M$ is exactly $\gcd(q-1,\alpha_{i+1}-\alpha_i)$  since
there are exactly that many $(\alpha_{i+1}-\alpha_i)$-th roots of
unity in $\Fq$ (the multiplicative group $\Fq^\times$ is cyclic with
$q-1$ elements). Since $p\nmid \alpha_{i+1}-\alpha_i$, each of these
roots lifts via Hensel lemma to a unique root of $g$ in $K^\ast$.
\end{proof}

\begin{prop}\label{lem6}
Let $p$ be an odd prime number and let $K$ be a finite extension of
$\Qp$ with ramification index~$e$ and residue class degree~$f$, such
that $p-1>e$, and set $q=p^f$.
 Then
\begin{align*}B_1(t,K)\leq  \big((t-1)D_1(t,K)+1\big)& (q-1)  \  \mbox{ and }
\\
& B_\m(t,K)\leq\big((t-1)D_\m(t,K)+1\big)(q-1).
\end{align*}
\end{prop}
\begin{proof}
We proceed as in the proof of Proposition~\ref{lem3}, grouping the roots
by valuation and by first digit. Let $f=a_0+a_1x^{\alpha_1} + \cdots
+ a_tx^{\alpha_t}\in K[x]$, with
$0=:\alpha_0<\alpha_1<\cdots<\alpha_t$, be a non-zero polynomial
with at most $t+1$ monomials. The Newton polygon $NP(f)$ of $f$ has
at most $t$ segments.

\smallskip
\noindent If the number of segments is bounded by $t-1$, then we
immediately get the bounds $$B_1(t,K)\leq(t-1)D_1(t,K)(q-1) \ \mbox{
and } \  B_\m(t,K)\leq(t-1)D_\m(t,K)(q-1),$$ since $B_1(t,\Fq)\le
q-1$, which are stronger than the bounds that we have to show.

\smallskip
\noindent Therefore we can assume that $NP(f)$ has exactly $t$
segments. In particular, $NP(f)$ consists of the segments
$(\alpha_i,v(a_i))-(\alpha_{i+1},v(a_{i+1}))$ for $0\le i\le t-1$.
If $p\,|\,\alpha_{i+1}-\alpha_i$ for  $0\le i\le t-1$  then
$p\,|\,\alpha _i$ for $1\le i\le t$ and  therefore $f(x)=g(x^p)$
where $g=a_0+a_1x^{\alpha_1/p} + \cdots + a_tx^{\alpha_t/p}$. The
roots of $f$ are the $p$-th roots of the roots of $g$. Since by
Lemma~\ref{raicesunidad} there is only one $p$-th root of  unity in
$K$, each root of $g$ gives at most one root of $f$, with the same
multiplicities.

\smallskip
\noindent Hence  we can reduce to the case where at least one of the
segments of $NP(f)$ satisfies $p\nmid\alpha_{i+1}-\alpha_i$.

\smallskip
\noindent In this case   Lemma~\ref{divis} implies that there are at
most $(q-1)$ roots of $f$ in $K^\ast$ with the valuation associated
to this segment, necessarily simple.  For the valuations
corresponding to the remaining $t-1$ segments, we have at most
$(t-1)D_1(t,K)(q-1)$ and $(t-1)D_\m(t,K)(q-1)$ roots of~$f$ counted
without/with multiplicities. This concludes the proof.
\end{proof}

As a consequence of Proposition~\ref{lem6}, we get the sharp bound
$B_1(1,K)=q-1$ for any finite extension $K/\Q_p$ with $p$ odd and
residue field of $q$ elements. The lower bound is attained by the
polynomial $x^{q-1}-1\in K[x]$.

\bigskip

Proposition~\ref{prop-binom} allows us to  prove the last result
needed in the proof of Theorem~\ref{thm3}.

\begin{prop}\label{thm9}
Let $K/\Qp$ be a finite extension with ramification index $e$ and
residue class degree $f$. Assume that $p>e+t$. Then
$$D_1(t,K)\leq
D_\m(t,K){\leq t}.$$
\end{prop}
\begin{proof}
As in the proof of Proposition~\ref{thm5}, it is enough to show that
given $\bfalpha$ and $\bfs$ s.t. $|\bfs|=t+1$,
$P_\bfalpha^\bfs(r_0,\dots,r_m) \neq 0$ for any distinct
$r_0,\ldots,r_m\in 1+\M$. Write $r_i=1+x_i$ with $x_i\in\M$ for
$0\le i\le m$, then  by Lemma~\ref{formula-vander1} and using
Notation~\ref{def-binom},
$$P_\bfalpha^\bfs(1+x_0,\dots,1+x_m)= \!\!\!\sum_{1\leq
\beta_1<\cdots<\beta_t}W_\bfbeta (\bfalpha)
\,P_\bfbeta^\bfs(x_0,\ldots,x_m).$$ The term of the right-hand side
corresponding to $\bfbeta={\bfst}=(1,\ldots,t)$ is equal to $W_\bfst
(\bfalpha)$,  since $P_{\bfst}^\bfs=1$, and is a non-zero integer
number  by Lemma~\ref{prop-binom}{\bf (b)} and
Observation~\ref{obsW}.

\noindent We  show that the remaining non-zero terms  for
$\bfbeta\neq{\bfst}$  have valuation strictly greater than
$v(W_\bfst (\bfalpha) )$: By Lemma~\ref{prop-binom}{\bf (c)}, their
ratio satisfies
$$\frac{W_\bfbeta(\bfalpha)\,P_\bfbeta^\bfs(x_0,\ldots,x_m) }{W_\bfst(\bfalpha)} \
= \ \frac{1!2!\cdots t!}{\beta_1!\cdots\beta_t!}\,
Q_\bfbeta(\bfalpha)\, P_\bfbeta^\bfs(x_0,\ldots,x_m),$$ where
$Q_\bfbeta(\bfalpha)\in\Z\setminus\{0\}$. Since $P_\bfbeta^\bfs$ is
homogeneous of degree $|\bfbeta|-t(t+1)/2$ and $v(x_i)\geq 1/e$ for
$0\le i\le s$, then
$$v\left(\frac{W_\bfbeta(\bfalpha)\,P_\bfbeta^\bfs(x_0,\ldots,x_m) }{W_\bfst(\bfalpha)}\right)
\geq \frac{|\bfbeta|-t(t+1)/2}{e}+v_p(1!2!\cdots
t!)-v_p(\beta_1!\cdots \beta_t!).$$ Our assumption $p>e+t$ implies
that $v_p(1!2!\cdots t!)=0$, so we can write
$$v\left(\frac{W_\bfbeta(\bfalpha)\,P_\bfbeta^\bfs(x_0,\ldots,x_m) }{W_\bfst(\bfalpha)}\right)
\geq\frac{1}{e}\sum_{i=1}^t\big(\beta_i-i-ev_p(\beta_i!)\big).$$
Since  $1\leq\beta_1<\cdots<\beta_t$ with $\bfbeta\ne \bfst$, then
$\beta_i\geq i$ for $1\le i\le t$ and there exists $j$ s.t.
$\beta_j>j$. We consider three cases:
\begin{itemize}
\item if $\beta_i<p$, then $\beta_i-i-ev_p(\beta_i!)=\beta_i-i\geq 0$.
\item if $p\leq\beta_i<2p$, then $\beta_i-i-ev_p(\beta_i!)\geq\beta_i-i-e\geq p-t-e>0$.
\item if $\beta_i\geq 2p$, then $\beta_i-i-ev_p(\beta_i!)\geq \beta_i-i-\frac{e\beta_i}{p-1}\geq 2p(1-\frac{e}{p-1})-t>2(p-1-e)-t\geq t>0$.
\end{itemize}
In all the cases we have $\beta_i-i-ev_p(\beta_i!)\geq 0$. Moreover,
when $\beta_j>j$, then   $\beta_j-j-ev_p(\beta_j!)>0$. This proves
that $$v\left(W_\bfbeta(\bfalpha)\,P_\bfbeta^\bfs(x_0,\ldots,x_m)
\right)-v\left(W_\bfst(\bfalpha)\right)>0$$ for any
$\bfbeta\neq{\bfst}$. In particular,
$v\left(P_\bfalpha^\bfs(r_0,\ldots,r_m)\right)=v\left(W_\bfst(\bfalpha)\right)$
which implies that $P_\bfalpha^\bfs(r_0,\ldots,r_m)\ne 0$ as
desired.
\end{proof}

\begin{proof}[Proof of Theorem~\ref{thm3}]
\begin{align*}B_\m(t,K)& \le \,\big((t-1)D_\m(t,K)+1\big)(q-1)\qquad \mbox{by Proposition~\ref{lem6}}\\
& \le \, \big((t-1)t+1\big)(q-1)\qquad \mbox{by
Proposition~\ref{thm9}}\end{align*}
\end{proof}

In~\cite{Lenstra}, H.W.~Lenstra introduced another technique to
produce upper bounds for $D_\m(t,L)$. For two non-negative integers
$t$ and $m$, he defines $d_t(m)$ to be the least common multiple of
all integers that can be written as the product of at most $t$
pairwise distinct positive integers that are at most $m$. Also for
any prime $p$, for any integer $t\geq 1$, and for any real number
$r>0$, he defines
$$C(p,t,r)=\max\big\{m\in\Z_{\geq 0}\,:\,mr-v_p(d_t(m))\leq\max_{0\le i\le t}\{ir-v_p(i!)\}\big\}.$$
In \cite[Thm.~3]{Lenstra} he proves that $D_\m(t,K)\leq C(p,t,1/e)$.
Next lemma shows that under the assumption $p>t+e$, we have
$C(p,t,1/e)=t$, therefore providing an alternative proof of
Proposition~\ref{thm9}.

\begin{lemma}\label{lem-lenstra}
Let $p$ be a prime number and let $t$ and $e$ be positive integers.
Assume that $p>t+e$. Then $C(p,t,1/e)=t$.
\end{lemma}
\begin{proof}
Observe that $d_t(t)=t!$, and then $C(p,t,1/e)\geq t$ by definition.
For the other inequality, we only have to show that for any $m > t$
and for any $i\leq t$,  $m/e-v_p(d_t(m))>i/e-v_p(i!)$ holds. By our
assumption on $p$, we clearly have $v_p(i!)=0$. Moreover, by
considering the same three cases analyzed during the proof of
Proposition~\ref{thm9}, we have $m-i>ev_p(m!)$. This concludes the
proof, since $v_p(m!)\geq v_p(d_t(m))$.
\end{proof}

% WHAT ABOUT THE GENERAL CASE, I.E. WHEN p<=t+e...
% CAN I IMPROVE LENSTRA'S BOUND ???

\section{Lower bounds}\label{sec-lower}

\begin{proof}[Proof of Example~\ref{thm4}]
Note first that $1$ is a double root, since $f(1)=f'(1)=0$ and $f''(1)=(q-1)^2(1+q^{q-1})q^{q-1}\neq 0$.
Also $q$ is an approximate root of $f$, since
$$f(q)=q^{2(q-1)}\left(q^{(q-1)(q^{q-1}-1)}-1\right)\;\;\Rightarrow\;\;v(f(q))=2(q-1)f_K$$
and also,
$$f'(q)=(q-1)(1+q^{q-1})q^{q-2}\left(q^{(q-1)q^{q-1}}-1\right)\;\;\Rightarrow\;\;v(f'(q))=(q-2)f_K.$$
Then $v(f(q))>2v(f'(q))$. Newton's method (Lemma~\ref{Hensel}) gives
an exact root $r\in K$ of~$f$ such that
$v(r-q)>v(f'(q))=(q-2)f_K\geq f_K=v(q)$. This implies that
$v(r)=v(q)=f_K$, and in particular $r\neq 1$. Note also that if
$x\in K$ is a root of $f$ and $\xi\in K$ is a $(q-1)$-root of the
unity (i.e.~$\xi^{q-1}=1$), then $f(x\xi)=0$, and similarly, if
$f'(x)=0$ then $f'(x\xi)=0$. Since there are exactly $q-1$ different
$(q-1)$-roots of  unity $\xi_1,\ldots,\xi_{q-1}\in K$, the
polynomial $f$ has $\xi_i$ as a double root and $r\xi_i$ as a simple
root for  $1\le i\le q-1$. This gives at least $3(q-1)$ roots
counted with multiplicities.
\end{proof}

\begin{proof}[Proof of Theorem~\ref{thm99}]
Let $A$ be the ring of integers of $K$ and let $\M$ be the maximal
ideal of $A$. We proceed by induction in $t$, proving a much
stronger statement: for any $t\geq 1$, there exists a polynomial
$f_t$, such that
\begin{enumerate}
\item $f_t\in\Zp[x]$,
\item $f_t$ is monic and it has non-zero constant term,
\item $f_t$ has $t+1$ terms,
\item $f_t$ has all exponents divisible by $q-1$,
\item $f_t$ has one simple root in $p^{t-1}(1+p\Zp)$,
\item $f_t$ has two simple roots in each $p^i(1+p\Zp)$ for $0\leq i<t-1$ if $t>1$,
\item $f_t$ has non-zero discriminant.
\end{enumerate}
By multiplying each of the roots of $f_t$ by the $(q-1)$-th roots of the unity in~$K$, we obtain at least
$(2t-1)(q-1)$ simple roots for $f_t$ in $K^\ast$. Items 3, 5 and 6 imply that the polynomial $f_t$ has a Newton
polygon with exactly $t$ segments (with slopes $0,-1,-2,\ldots,-t+1$), and therefore all its roots (even the
ones in $\overline{K}$) have necessarily valuation $0,1,\ldots,t-1$.

\smallskip
\noindent  The polynomial $f_1=x^{q-1}-1$ proves the case $t=1$. Now
assume that $f_t=x^{\alpha_t}+a_{t-1}x^{\alpha_{t-1}} +
\cdots+a_1x^{\alpha_1}+a_0\in\Zp[x]$ satisfies Conditions 1-7. Since
$f_t$ is monic with coefficients in~$\Zp$, all its roots in $K$
belong to $A$, and in particular $f_t(1/p)\neq 0$. Furthermore
$f_t(1/p)\in p^{-\alpha_t}(1+p\Zp)$, and therefore
$$\hat{f}_t(x):= \frac{f_t(x/p)}{f_t(1/p)} \ = \ \frac{p^{\alpha_t}f_t(x/p)}{p^{\alpha_t}f_t(1/p)} \ = \ u^{-1}\,{h(x)}
$$  where
\begin{align}\label{hx}& h(x):=
x^{\alpha_t}+a_{t-1}p^{\alpha_t-\alpha_{t-1}}x^{\alpha_{t-1}}+\cdots+a_0p^{\alpha_t}
\ \in \ \Z_p[x],\\ \nonumber &
u:=1+a_{t-1}p^{\alpha_t-\alpha_{t-1}}+\cdots+a_0p^{\alpha_t} \ \in \
1+p\Zp.\end{align} Therefore $\hat{f}_t(x)\in \Z_p[x]$ and we define
$$g_\alpha(x):=x^\alpha-\hat{f}_t(x)\in\Zp[x]$$
for $\alpha>\alpha_t$.  We show that, for suitable $\alpha>\alpha_t$
and $\varepsilon\in\Zp$, the polynomial
$f_{t+1}(x)=g_\alpha(x)+\varepsilon$ satisfies Conditions 1--7 for
$t+1$:

\smallskip
\noindent  Since $\hat{f}_t(0)\ne 0$, then
 $g_\alpha$ satisfies Conditions 1--3
for any $\alpha>\alpha_t$. In addition $g_\alpha(1)=0$ by
construction.

\smallskip
\noindent We remark that since $f_t$ and $g_\alpha$ are monic in
$\Zp[x]$, then all their roots in $\Qp$ belong to  $\Zp$. Define
$\gamma_t=\max\{v(f_t'(r))\,:\,r\in \Zp\, , \,f_t(r)=0\}$. Note that
$\gamma_t\neq\infty$ because $f_t$ has non-zero discriminant.

\smallskip \noindent
Assume  $\alpha\geq 2(\gamma_t+\alpha_t)$. We
 prove first that if
$r_0\in\Zp$ is a root of $f_t$, then $p\,r_0$ is an approximate root
of~$g_\alpha$, which induces a root $r\in\Zp$ of~$g_\alpha$
with $v(r)=v(r_0)+1$:\\
The condition $f_t(r_0)=0$ implies  $$g_\alpha(p\,r_0)=p^\alpha
r_0^\alpha \ \ \mbox{and} \ \ g'_\alpha(p\,r_0)=\alpha
p^{\alpha-1}r_0^{\alpha-1}-f_t'(r_0)/(pf_t(1/p)).$$
 Since $v(\alpha
p^{\alpha-1}r_0^{\alpha-1})\geq\alpha-1$ and
$v(f_t'(r_0)/(pf_t(1/p)))\leq\gamma_t+\alpha_t-1<\alpha/2$, then
$v(g_\alpha'(p\,r_0))<\alpha/2\leq v(g_\alpha(p\,r_0))/2$,
Lemma~\ref{Hensel} implies that $p\,r_0$ is an approximate root of
$g_\alpha$, corresponding to   a root  $r\in \Z_p$. Moreover,
$v(r-p\,r_0)>\alpha/2$, which implies $v(r-p\,r_0)>\alpha_t\geq
t\geq v(p\,r_0)$ by the observation after Conditions 1--7,  and in
particular $v(r)=v(p\,r_0)=v(r_0)+1$. Therefore  each root $r_0\in
p^i(1+p\Zp)$ for $0\leq i\le t-1$ satisfying  Conditions 5 or 6 of
$f_t$ induces a simple root $r\in p^{i+1}(1+p\Zp)$ of $g_\alpha$. We
still need to show these are all different.

\smallskip \noindent
Define $\gamma_t'=1+\max\{v(r_0-r'_0)\,:\,r_0,r'_0\in\Zp\, ,
\,f_t(r_0)=f_t(r'_0)=0 \ \mbox{and} \ r_0\neq r'_0\}$ and assume
that $\alpha\geq \max\{2(\gamma_t+\alpha_t),2\gamma_t'\}$. Then two
different roots $r_0\ne r'_0$ of $f_t$ in $\Z_p$ induce  different
roots $r\ne r'$ of  $g_\alpha$ in $\Zp$, since if $r=r'$ then
$$1+v(r_0-r'_0)=v(p\,r_0-p\,r'_0)\geq\min\{v(r-p\,r_0),
v(r'-p\,r'_0)\}>\alpha/2\geq\gamma_t',$$ in contradiction with the
definition of~$\gamma_t'$.

\smallskip \noindent
Therefore, we proved so far that for
$\alpha>2(\alpha_t+\gamma_t+\gamma_t')$,   $g_\alpha$ has at least
one simple root in $p^t(1+p\Z_p)$, two simple roots in each
$p^i(1+p\Z_p)$ for $1\le i<t$ and the root $1\in 1+p\Zp$.

\smallskip \noindent Our aim now is to produce an extra root. We    construct such a root
 in $1+p\Zp$ but different from $1$ following the following strategy. We
 start with a fixed $r_0$  congruent to $1$ modulo $p$ but not
 congruent to $1$ modulo $p^2$, and show that we can guarantee the
 existence of some $\alpha$ such that the conditions of
 Lemma~\ref{Hensel} are satisfied for $r_0$ and $g_\alpha$. In order
 to achieve this, we construct a sequence of exponents
 $\alpha^{(i)}$ such that the order of $r_0$ as a root increases.

 \smallskip
\noindent
 Fact~1 below shows that there exists $r_0$ with the required
 conditions such that $\big(v(g_{\alpha^{(i)}}'(r_0))\big)$ is bounded.
 Assuming this holds,  we can pick a large enough $i>>1$ such that $g=g_{\alpha^{(i)}}$ satisfies Conditions 1--6 in our list
for $t+1$. Let $r_1,\ldots,r_{2t+1}{\in\Zp}$ be the ${2t+1}$ simple
roots of $g$ of Conditions 5 and 6 and set  $C:=2\max\{t,
v(g'(r_j)), \,1\le j\le 2t+1\}$.

\smallskip
\noindent We will define       $f_{t+1}(x):=g(x)+\varepsilon$ for
some $\varepsilon\in\Zp$ so that $f_{t+1}$ satisfies Conditons 1--7
for $t+1$. By Lemma~\ref{Hensel}, if $v(\varepsilon)>C$, then
$v(f_{t+1}(r_j))=v(\varepsilon)>2v(f'_{t+1}(r_j))$ for any $j$.
Therefore the roots $r_1,\ldots,r_{2t+1}$ are approximate roots of
$f_{t+1}$, with corresponding induced roots
$\hat{r}_1,\ldots,\hat{r}_{2t+1}\in\Zp$ that are all different and
satisfy
$$v(\hat{r}_j-r_j)\ge v(\varepsilon)-v(f'_{t+1}(r_j))>C/2 \ge t \ge v(r_j),$$
which implies $v(\hat r_j)=v(r_j)$.  Also, if $v(\varepsilon)>
v(g(0))$, then $f_{t+1} $ has a non-zero constant term. Finally, the
discriminant of $f_{t+1}$ is a polynomial in $\varepsilon$ of
positive degree, and therefore vanishes at finitely many values of
$\varepsilon$. We conclude by selecting $\varepsilon$ with
$v(\varepsilon)>\max\{C,v(g(0))\}$ such that $f_{t+1}$ has non-zero
discriminant. This  polynomial $f_{t+1}$ satisfies Conditions 1--7.

\smallskip \noindent
The rest of the proof focuses now on guaranteeing the existence of
such an $r_0$ and such a sequence $(\alpha^{(i)})_i$. Let $r_0$ be
any element of $ 1+p\Zp$ such that $r_0\not\equiv 1\pmod{p^2}$.
Therefore $p^2\nmid r_0^{p-1}-1$ and   $\hat{f}_t(r_0)\in 1+p\Zp$.
Lemma~\ref{akp} below implies there exists a sequence of integers
$(\alpha^{(i)})_{i\geq 1}$ satisfying for all $i$:
\begin{itemize}
\item
 $\alpha^{(1)}\equiv 0 \pmod{\varphi(p)}$ and
$\alpha^{(i+1)}\equiv\alpha^{(i)}\pmod{\varphi(p^i)}$,
\item
$r_0^{\alpha^{(i)}}\equiv \ \hat{f}_t(r_0) \pmod{p^i}$, \quad  i.e.
\quad  $g_{\alpha^{(i)}}(r_0)\equiv 0 \pmod{p^i}$,
\item $q-1\,|\, \alpha^{(i)}$ and $\alpha^{(i)} \ge 2(\alpha_t+\gamma_t+\gamma_t')$.
\end{itemize}
Since $p^i \mid g_{\alpha^{(i)}}(r_0)$, then
 $v(g_{\alpha^{(i)}}(r_0))\geq i$ for all $i\in\N$.
By Fact 1 at the end of this proof, there exists some $r_0\in
1+p\Zp$ with $p^2\nmid r_0^{p-1}- 1$ such that the sequence
$\big(v(g_{\alpha^{(i)}}'(r_0))\big)_{i\ge 1}$  is bounded.
 Therefore,  fixing  $\alpha^{(i)}$ big enough, the hypotheses
 of Lemma~\ref{Hensel} are satisfied, and
$r_0$ is an approximate root of $g_{\alpha^{(i)}}$ inducing a root
$r\in \Z_p$ with  $v(r-r_0)>v(g_{\alpha^{(i)}}'(r_0))$.
\\Now for $i\ge 2$, since $r_0\equiv 1 \pmod p $ and  by Fact~2
below, $\alpha^{(i)}\equiv \alpha_t \pmod p$, we have
$$g_{\alpha^{(i)}}'(r_0) \equiv
g_{\alpha^{(i)}}'(1)\equiv \alpha^{(i)}-\alpha_t\equiv 0\pmod{p}, \
\ \mbox{i.e.} \ \ v\big(g_{\alpha^{(i)}}'(r_0)\big)\ge 1,$$ and
therefore $v(r-r_0)>1$, which implies $r=(r-r_0)+r_0\in 1+p\Zp$ and
$r\equiv r_0\not\equiv 1\pmod{p^2}$. In particular $r_0\neq 1$ is a
second simple root of $g_{\alpha^{(i)}}$ in $1+p\Z_p$.

\medskip
\noindent

\medskip
\noindent {\em Fact 1.} \ The sequence
$\big(v(g_{\alpha^{(i)}}'(r_0))\big)_{i\ge 1}$  is bounded for some
$r_0\in 1+p\Zp$ such that $p^2\nmid r_0^{p-1}- 1$.

\smallskip
\noindent {\em Proof of Fact 1.}  Assume it is not for any $r_0$
satisfying the hypotheses, then we can extract a subsequence
$(\beta_j)_{j\geq 1}$, where $\beta_j=\beta_j(r_0)$,
  of
$(\alpha^{(i)})_{i\geq 1}$ with $\beta_1\equiv 0 \pmod{\varphi( p)}$
such that for all $j$, \begin{align*} &  \beta_{j+1}\equiv
\beta_j\pmod{\varphi(p^j)}\ , \ r_0^{\beta_j}\equiv \ \hat{f}_t(r_0)
\pmod{p^j}\\ & \mbox{ and }\ v(g_{\beta_j}'(r_0))\ge j, \ \mbox{
i.e. } \ \beta_jr_0^{\beta_j-1}\equiv \hat f_t'(r_0) \pmod{p^j}\\
& \qquad \qquad \qquad \mbox{ or equivalently }
\beta_jr_0^{\beta_j}\equiv  r_0\,\hat f_t'(r_0) \pmod{p^j} .
\end{align*}
The sequence $(\beta_j)_{j\geq 1}$ has by construction a limit $\beta$ in
the set
$$\EE_p=\ilim \Z/\varphi(p^n)\Z\approx\Z/(p-1)\Z\oplus\Zp \ , \ \beta\mapsto \big(\beta_1=0,(\beta_j)_{j\ge 2}\big)$$ of
$p$-adic exponents, as defined in~\cite[Def.~2.1]{AKP06}.\\
Thus, for all $r_0\in 1+p\Zp$ such that $p^2\nmid r_0^{p-1}- 1$
there exists $\beta:=\beta(r_0)\in \EE_p$ such that
$$r_0^\beta=\hat{f}_t(r_0)\ \mbox{ and } \
\beta\hat{f}_t(r_0)=r_0\hat{f}'_t(r_0)  \ \mbox{ in } \ \Zp$$ where
the exponential of an element of $\Zp^\times$ by an element of
$\EE_p$ is defined in~\cite[Prop.~2.2]{AKP06}. Our goal is to prove
that if this is the case, then $\hat{f}_t$ needs to be a monomial,
that is, $\hat f_t=a x^\gamma$ for some $a\in \Q_p$ and $\gamma \in
\N$. But clearly $\hat f_t$ is not a monomial by construction,
giving a contradiction. Therefore this would prove  Fact~1.

   \smallskip \noindent
 Given such an $r_0$, let us
define $r_N= r_0+p^N$ for $N\geq 2$, which satisfies the same
conditions, and denote $\beta:=\beta(r_0)$ and
$\beta_N:=\beta(r_N)$. Then
\begin{align*}
 \beta_N\hat{f}_t(r_N)=r_N\hat{f}'_t(r_N) &\ \Rightarrow
\ \beta_N \hat{f}_t(r_0)\equiv r_0      \hat{f}'_t(r_0) \pmod{p^N}
\\ &\ \Rightarrow \ \beta_N \equiv\beta \pmod{p^N}.\end{align*}
Therefore, since $p-1\mid \beta$ for any $\beta$,  $\beta_N\equiv
\beta \pmod{\varphi(p^{N+1})}$ and  we can write $$\beta_N=\beta
+\varphi(p^{N+1})\,\delta \ \mbox{  for some   } \ \delta \in \Zp.$$
Now, Taylor expanding    $\hat{f}_t(r_0+p^N)$ around $r_0$ up to
order $p^{2N}$ we obtain
\begin{align*}
&(r_0+p^N)^{\beta+\varphi(p^{N+1})\,\delta}=\hat f_t(r_0+p^N) \ \Longrightarrow \\
&r_0^\beta(1+p^Nr_0^{-1})^{\beta}\left(r_0^{\varphi(p^{N+1})}(1+p^Nr_0^{-1})^{\varphi(p^{N+1})}
\right)^{\delta}\equiv
\hat{f}_t(r_0)+p^N\hat{f}'_t(r_0)\pmod{p^{2N}}.\end{align*} We write
$r_0=1+p\,x_0$ and therefore, since
$$r_0^{\varphi(p^{N+1})}=(1+p\,x_0)^{(p-1)p^{N}}=1+p^{N+1}u_N(r_0)$$
for some $u_N(r_0)\in\Zp$, we get
$$\hat{f}_t(r_0)(1+p^Nr_0^{-1}\,\beta)(1+p^{N+1}u_{N}(r_0)\,\delta)\equiv \hat{f}_t(r_0)+p^N\hat{f}'_t(r_0)
\pmod{p^{2N}}.$$ Substracting $\hat f_t(r_0)$, multiplying by $r_0$
and dividing by $p^N$ gives
$$\hat{f}_t(r_0)\,\beta + p\,r_0u_N(r_0)\hat{f}_t(r_0)\,\delta\equiv r_0\hat{f}'_t(r_0)\pmod{p^N},$$
i.e., since  $\hat{f}_t(r_0)\,\beta=r_0\hat{f}'_t(r_0)$, we obtain
$r_0\,u_N(r_0)\hat{f}_t(r_0)\,\delta\equiv 0\pmod{p^{N-1}}$.

\noindent Now we observe that
$$u_{N}(r_0)=(r_0^{\varphi(p^{N+1})}-1)/p^{N+1}\equiv
(r_0^{p-1}-1)/p\not\equiv 0\pmod{p},$$ and therefore since
$r_0\equiv 1\pmod p$ and $\hat f_t(r_0)\equiv 1 \pmod p$, we
conclude that $\delta\equiv 0\pmod{p^{N-1}}$, i.e. $\beta_N\equiv
\beta \pmod{p^{2N-1}}$.  Going back to the identity $\beta_N\hat
f_t(r_N)=r_N\hat f'_t(r_N)$ and Taylor expanding now around $r_0$ up
to order $p^{2N-1}$  we obtain
$$\beta\,\hat{f}_t(r_0)+p^N\,\beta\,\hat{f}'_t(r_0)\equiv (r_0+p^N)\hat{f}'_t(r_0)+ p^N\,r_0\,\hat{f}''_t(r_0)
\pmod{p^{2N-1}},$$ which simplifies to
$$(\beta-1)\,\hat{f}'_t(r_0)\equiv
r_0\,\hat{f}''_t(r_0)\pmod{p^{N-1}}, \quad \forall \ N\ge 2,$$
   and therefore
$(\beta-1)\hat{f}'_t(r_0)=r_0\hat{f}''_t(r_0)$ in $\Zp$.

\smallskip \noindent This last identity
combined with $\beta\hat{f}_t(r_0)=r_0\hat{f}'_t(r_0)$ implies the
following differential equation independent from $\beta$:
$$ r_0\,\hat{f}''_t(r_0)\,\hat f_t(r_0)+\hat{f}_t(r_0)\,\hat{f}'_t(r_0)-r_0\hat{f}'_t(r_0)^2=0.$$
Since this identity holds for infinitely many $r_0\in \Z_p$, it is a
polynomial identity in $\Q_p[x]$ that can be rewritten as
$$\big(x\,\hat{f}'_t(x)/\hat{f}_t(x)\big)'=0.$$ This means that
$x\,\hat{f}'_t(x)=\gamma \,\hat{f}_t(x)$ for some $\gamma\in \Qp$,
and then  if $\hat f_t\ne 0$, then $\gamma\in \N$ and $\hat f_t= a
x^\gamma$ is a monomial.   This proves Fact 1.

\medskip \noindent  {\em Fact 2.} \  $\alpha^{(i)}\equiv\alpha_t\pmod{p}$ for
all $i\geq 2$.

\smallskip \noindent  {\em Proof of Fact 2.} We note that,   since $2\le q-1
\mid \alpha_j$ for all $j$,  Formula~(\ref{hx}) implies that in
$\Z_p[x]$,
$$h(x)\equiv x^{\alpha_t} \pmod{p^2} \ \mbox{ and }  \  u \equiv
1 \pmod{p^2} \ \Longrightarrow  \  u^{-1}\equiv 1 \pmod{p^2}.$$
Therefore $$\hat f_t(x)\equiv x^{\alpha_t} \pmod{p^2}.$$ Thus
writing $r_0=1+p\,x_0$ with $x_0\in\Zp$ and $p\nmid x_0$ we get
$$\hat{f}_t(r_0) \equiv r_0^{\alpha_t}\equiv(1+p\,
x_0)^{\alpha_t}\equiv 1+\alpha_t p\, x_0\pmod{p^2}.$$ On the other
hand, by construction, for any $i\ge 2$,
$$ \hat f_t(r_0)\equiv r_0^{\alpha^{(i)}}\equiv (1+p\, x_0)^{\alpha^{(i)}}\equiv 1 + {\alpha^{(i)}}p\,  x_0 \pmod{p^2}.$$
 This implies
$\alpha^{(i)}\equiv\alpha_t\pmod{p}$ since $p\nmid x_0$, and proves
Fact~2.
\end{proof}

\begin{lemma} \label{akp} Let $y\equiv 1\pmod p$ in $ \Z_p[x]$ and let $r\in \Zp$ be such that
$r\equiv 1 \pmod{p}$ and  $r\not\equiv 1\pmod{p^2}$. Then, given
$f,C\in \N$,  there exists a sequence of natural numbers
$(\alpha^{(i)})_{i\ge 1}$ satisfying that for all $i\in \N$,
\begin{itemize}
\item
 $\alpha^{(1)}\equiv 0 \pmod{\varphi(p)}$ and
$\alpha^{(i+1)}\equiv\alpha^{(i)}\pmod{\varphi(p^i)}$,
\item
$r^{\alpha^{(i)}}\equiv y \pmod{p^i}$,
\item $p^f-1\,|\, \alpha^{(i)}$ and $\alpha^{(i)} \ge C$.
\end{itemize}
\end{lemma}
\begin{proof} We apply \cite[Proposition~3]{AKP06}  to
$g=r^\alpha$: \\
Since $r^0\equiv y \pmod p $ and $r\not\equiv 1\pmod{p^2}$ implies
$r^{p-1}\not \equiv 1 \pmod{p^2}$, then  there exists a sequence
$0=:\beta_1,\beta_2,\dots$ such that $\beta_{i+1}\equiv
\beta_i\pmod{\varphi(p^i)}$ and $r^{\beta_{i}}\equiv y
\pmod{p^{i}}$ for all $i$. \\
Now we show that there exists $k_i\in \N$ such that
$\alpha^{(i)}:=\beta_i+ k_i\varphi(p^i)$ satisfies all the
conditions. First we observe that under those conditions, since
$r\not \equiv 0 \pmod p$, then
\begin{align*}r^{\alpha^{(i+1)}}&\equiv \ r^{\beta_{i+1}}    \  \equiv  \  y \pmod{p^{i+1}}
\\ &\  \equiv \ r^{\beta_i}
 \  \equiv \  r^{\alpha^{(i)}} \pmod{p^i}\end{align*}
Therefore we only need to show that some $k_i$ satisfies the last
conditions. The congruence equation $\beta_i+ k_i\varphi(p^i) \equiv
0 \pmod{(p^f-1)}$ is equivalent, since $\beta_i\equiv 0
\pmod{(p-1)}$, to the equation
$$k_i\, p^{i-1} \equiv - \beta_i/(p-1) \pmod{(1+\cdots +p^{f-1})}$$
which solutions exist and are equal to $k_{i,0} + k (1+\cdots
+p^{f-1})$ for all $k\in \Z$, where $k_{i,0}$ is a particular
solution,  since $\gcd(p^{i-1} ,1+\cdots +p^{f-1})=1$. Clearly $k$
can be chosen big enough so that $\alpha^{(i)} \ge C$.
\end{proof}

\section*{Acknowledgements}
It is a pleasure to thank the anonymous referee for  several
interesting comments, suggestions and examples. Martin Avenda\~no
also thanks Maurice Rojas, Korben Rusek and Ashraf Ibrahim for many
fruitful discussions on upper and lower bounds of univariate sparse
polynomials over $p$-adic fields.

\end{document}